\documentclass[11pt, reqno]{amsart}

\usepackage{amscd,amssymb,amsopn,amsmath,amsthm,mathrsfs,graphics,amsfonts,enumerate,verbatim,calc,
}
\usepackage{mathrsfs}
\usepackage[all]{xy}
\newcommand\redout{\bgroup\markoverwith
{\textcolor{red}{\rule[.5ex]{2pt}{0.4pt}}}\ULon}
\usepackage{xcolor}
\usepackage[colorlinks=true,linkcolor=blue,citecolor=blue]{hyperref}
\usepackage{color}

\usepackage[OT2,OT1]{fontenc}
\newcommand\cyr{%
\renewcommand\rmdefault{wncyr}%
\renewcommand\sfdefault{wncyss}%
\renewcommand\encodingdefault{OT2}%
\normalfont
\selectfont}
\DeclareTextFontCommand{\textcyr}{\cyr}

\usepackage{amssymb,amsmath}

\usepackage[top=1in, bottom=1in, right=1in, left=1in]{geometry}
\newtheorem{theorem}{Theorem}[section]
\newtheorem{lemma}[theorem]{Lemma}
\newtheorem{proposition}[theorem]{Proposition}

\newtheorem{corollary}[theorem]{Corollary}

\newtheorem{question}[theorem]{Question}

\newtheorem{conjecture}[theorem]{Conjecture}

\theoremstyle{definition}
\newtheorem{definition}[theorem]{Definition}
\newtheorem{remark}[theorem]{Remark}

\theoremstyle{remark}

\newtheorem{example}[theorem]{Example}

\newcommand{\fm}{\frak{m}}
\newcommand{\m}{\frak{m}}

\begin{document}

\title[The first tight Hilbert coefficient]{The vanishing of the first tight Hilbert coefficient for Buchsbaum rings}
\author[Duong Thi Huong]{Duong Thi Huong}
\address{Department of Mathematics, Thang Long University, Hanoi, Vietnam}
\email{huongdt@thanglong.edu.vn}

\author[Nguyen Tuan Long]{Nguyen Tuan Long}
\address{Department of Mathematics, National Economics University, Hanoi, Vietnam}
\email{ntlong81@gmail.com}

\author[Pham Hung Quy]{Pham Hung Quy${}^*$}
\address{Department of Mathematics, FPT University, Hanoi, Vietnam}
\email{quyph@fe.edu.vn}
\thanks{$^*$Corresponding author: Pham Hung Quy
}

\thanks{This work is partially supported by a fund of Vietnam National Foundation for Science
and Technology Development (NAFOSTED) under grant number 101.04-2023.08.}

\maketitle

\begin{abstract}
We prove that if the first tight Hilbert coefficient vanishes then ring is $F$-rational provided it is a Buchsbaum local ring satisfying the $(S_2)$ condition.  \\
{\em 2020 Mathematics Subject Classification}: 13A35, 13H10, 13D40, 13D45.\\
Keywords: {tight Hilbert coefficients, Buchsbaum rings.}

\end{abstract}

\maketitle

\thispagestyle{empty}

\section{Introduction}
Let $(R,\m)$ be a Noetherian local ring of dimension $d$ such that $\widehat{R}$ is reduced and let $I\subseteq R$ be an $\m$-primary ideal. Then for $n\gg0$, $\ell(R/\overline{I^{n+1}})$ agrees with a polynomial in $n$ of degree $d$, and we have integers $\overline{e}_0(I),\dots, \overline{e}_d(I)$ such that
$$\ell(R/\overline{I^{n+1}})=\overline{e}_0(I)\binom{n+d}{d} - \overline{e}_1(I)\binom{n+d-1}{d-1} + \cdots + (-1)^d\overline{e}_d(I).$$
These integers $\overline{e}_i(I)$ are called the normal Hilbert coefficients of $I$.

It is well-known that $\overline{e}_0(I)$ is the Hilbert-Samuel multiplicity of $I$, which is always a positive integer. The first coefficient $\overline{e}_1(I)$ is sometimes called the normal Chern coefficient of $I$. It was proved by Goto-Hong-Mandal \cite{GHM} that when $\widehat{R}$ is unmixed, $\overline{e}_1(I)\geq 0$ for all $\m$-primary ideals $I\subseteq R$ (which answers a question posed by Vasconcelos \cite{V}). If $(R,\m)$ is a Noetherian local ring such that $\widehat{R}$ is reduced and $(S_2)$, then $\overline{e}_1(Q)=0$ for some parameter ideal $Q\subseteq R$ if and only if $R$ is regular and $\nu(\m/Q)\leq 1$, see \cite{MQ24}.

In \cite{GMV}, it was shown that when $R$ is reduced ring of characteristic $p>0$, for $n\gg0$, the function $\ell(R/(I^{n+1})^*)$ also agrees with a polynomial of degree $d$, and one can define the tight Hilbert coefficients $e_0^*(I),\dots, e_d^*(I)$ in a similar way (see Section 2 for more details). It is easy to see that $\overline{e}_1(I)\geq e_1^*(I)$. Recently, Ma-Quy proved that $e_1^*(Q)\geq 0$ for any parameter ideal $Q\subseteq R$ under mild assumptions, see \cite[Theorem 1.2]{MQ24}. The following conjecture was asked by Huneke in \cite{DQV, MQ24}.
\begin{conjecture} If $R$ is excellent and has characteristic $p>0$ (such that $\widehat{R}$ is reduced and $(S_2)$), and $e_1^*(Q) = 0$ for some (and hence for all) parameter ideal $Q\subseteq R$, then $R$ is $F$-rational.
\end{conjecture}
This conjecture is true for Cohen-Macaulay local rings in \cite[Theorem 4.4]{GMV}, and tight Buchsbaum local rings in \cite[Corollary 4.6]{DQV}. This paper proves the conjecture under new assumptions; the main theorem is as follows. 
\begin{theorem}
	Conjecture 1 holds if we further assume that R is reduced and Buchsbaum.
\end{theorem}
We also prove that the condition $(S_2)$ of the ring cannot be omitted in Proposition \ref{Pro1}.

\subsection*{Acknowledgement} The authors are grateful to the referee for carefully reading the paper. This article was written during many visits of the third author to Vietnam
Institute for Advanced Studies in Mathematics. He thanks sincerely the institute
for their hospitality and valuable supports. The third author also thanks to Basque Center for Applied Mathematics for their partial support.  He is grateful to Kriti Goel, Linquan Ma and Ilya Smirnov for their comments on the paper.

\section{Preliminaries}
\subsection{Hilbert coefficients}
Let $(R, \frak m)$ be a Noetherian local ring of dimension $d$ and let $I\subseteq R$ be an $\m$-primary ideal. Then for all $n \gg 0$ we have
$$\ell(R/I^{n+1}) = e_0(I) \binom{n+d}{d} - e_1(I) \binom{n+d-1}{d-1} + \cdots + (-1)^d e_d(I),$$
where $e_0(I), \cdots, e_d(I)$ are all integers, and are called the Hilbert coefficients of $I$.

Now we suppose $\widehat{R}$ is reduced. This condition guarantees that $R \oplus \overline{I}t \oplus \overline{I^2}t^2 \oplus \cdots $ is module-finite over $R[It]$. As a consequence of this fact, one can show that for all $n\gg0$, $\ell(R/\overline{I^{n+1}})$ agrees with a polynomial in $n$ and one can write

$$\ell(R/\overline{I^{n+1}}) = \overline{e}_0(I) \binom{n+d}{d} - \overline{e}_1(I) \binom{n+d-1}{d-1} + \cdots + (-1)^d \overline{e}_d(I),$$
where the integers $\overline{e}_0(Q), \cdots, \overline{e}_d(Q)$ are called the normal Hilbert coefficients. It is well-known that $e_0(I)=\overline{e}_0(I)$ agrees with the Hilbert-Samuel multiplicity $e(I, R)$ of $I$. 

Now in characteristic $p >0$, the tight closure of an ideal $I \subseteq R$, introduced by Hochster-Huneke, is defined as follows:
$$I^* :=\{x \in R \mid \text{there exists } c \in R \setminus \cup_{\frak p \in \text{min}(R)} \frak p \text{ such that } cx^{p^e} \in I^{[p^e]} \text{ for all } e\gg 0\}.$$
A local ring $(R, \frak m)$ of prime characteristic $p > 0$ is called $F$-rational if
every ideal generated by a system of parameters is tightly closed, i.e. $Q^* = Q$ for every parameter ideal $Q$. In general, tight closure is always contained in the integral closure, i.e. $I^* \subseteq \bar{I}$. We know that $R\oplus I^*t \oplus (I^2)^*t^2 \oplus \cdots$ is an $R[It]$-algebra that is also module-finite over $R[It]$: the fact that it is an $R[It]$-algebra follows from the fact that $(I^a)^*(I^b)^*\subseteq (I^{a+b})^*$ for all $a, b$, and that it is module-finite over $R[It]$, it is an $R[It]$-subalgebra of $R \oplus \overline{I}t \oplus \overline{I^2}t^2 \oplus \cdots $, and the latter is module-finite over $R[It]$ (note that $R[It]$ is Noetherian). Based on the discussion above, one can show that for all $n\gg0$, $\ell(R/(I^{n+1})^*)$ also agrees with a polynomial in $n$, and we write
$$\ell(R/(I^{n+1})^*) = e_0^*(I) \binom{n+d}{d} - e_1^*(I) \binom{n+d-1}{d-1} + \cdots + (-1)^d e_d^*(I),$$
for all $n\gg 0$ (see \cite{HM} for more general results). We call the integers $e_0^*(I),\dots, e_d^*(I)$ the tight Hilbert coefficients with respect to $I$. It is easy to see that $e_0^*(I)=e(I, R)$ is still the Hilbert-Samuel multiplicity of $I$, and that we always have $\overline{e}_1(I)\geq e_1^*(I)\geq e_1(I)$ by comparing the coefficients of $n^{d-1}$ and note that $I^n\subseteq(I^n)^*\subseteq \overline{I^n}$ for all $n$. The main result of this paper is inspired by the following results.
\begin{theorem}[\cite{V, MSV, GGH10}] Let $Q$ be a parameter ideal of $R$. Then $e_1(Q) \le 0$. Moreover, assume $R$ is unmixed, $R$ is Cohen-Macaulay if and only if $e_1(Q) = 0$ for some $Q$.
\end{theorem}
\begin{theorem}[\cite{GHM, MQ24}] Let $(R, \frak m)$ be a local ring such that $\widehat{R}$ is reduced and equidimensional. Then for all parameter ideals $Q$ we have $\overline{e}_1(Q) \ge 0$. Moreover, if $\overline{e}_1(Q) = 0$ for some parameter ideal $Q$, then $R$ is regular and $\nu(\frak m/Q) \le 1$ provided $\widehat{R}$ is $(S_2)$.  
\end{theorem}

\begin{theorem}[\cite{MQ24, GMV}] \label{CM case} Let $(R, \frak m)$ be a local ring of prime characteristic $p$ such that $\widehat{R}$ is reduced and equidimensional. Then for all parameter ideals $Q$ we have $e^*_1(Q) \ge 0$. Assume $R$ is Cohen-Macaulay, then $e^*_1(Q) = 0$ for some parameter ideal $Q$ if and only if $R$ is $F$-rational.
\end{theorem}
\subsection{Buchsbaum and tight Buchsbaum rings}
We recall the definition of Buchsbaum rings.
\begin{definition}
A local ring $(R, \frak m)$ is called {\it Buchsbaum} if $\ell(R/Q)- e_0(Q)$ is independent of the choice of parameter ideal $Q$.
\end{definition}
\begin{remark}
Let $(R, \frak m)$ be a Buchsbaum local ring of dimension $d$ and $Q$ any parameter ideal (for example, see \cite{T}). We have
$$\ell(R/Q) -  e_0(Q) = \sum_{i=0}^{d-1} \binom{d-1}{i} \ell(H^i_{\frak m}(R)).$$
\end{remark}
The Hilbert function of parameter ideals in Buchsbaum rings can be understood in terms of local cohomology, see \cite[Corollary 3.2]{S}, \cite[Corollary 4.2]{T}. 
\begin{theorem}\label{Hil Buchs} Let $R$ be Buchsbaum and $Q$ any parameter ideal of $R$. Then for $n\geq 0,$ 
$$\ell(R/Q^{n+1})=\sum_{i=0}^d(-1)^ie_i(Q)\binom{n+d-i}{d-i}, \text{ where }$$
$$e_i(Q)=(-1)^i\sum_{j=0}^{d-i}\binom{d-i-1}{j-1} \ell(H^j_{\fm}(R) )\mbox{ for all } i=1,2,\ldots, d.$$
\end{theorem}

Recently, Ma and the third author in \cite{MQ22} considered the analogy of Buchsbaum rings in characteristic $p$ as follows.
\begin{definition}
An unmixed local ring $(R, \frak m)$ of characteristic $p >0$ is called {\it tight Buchsbaum} if $e_0(Q) - \ell(R/Q^*)$ is independent of the choice of parameter ideal $Q$.
\end{definition}
The following result is proved in \cite[Theorem 4.3]{MQ22} (see also \cite{Q17} for partial result). 
\begin{remark}
Let $(R, \frak m)$ be a tight Buchsbaum local ring of dimension $d$ and $Q$ any parameter ideal. We have
$$ \ell(R/Q)- \ell(R/Q^*) = \sum_{i=0}^{d-1} \binom{d}{i} \ell(H^i_{\frak m}(R)) + \ell(0^*_{H^d_{\frak m}(R)}),$$
and 
$$e_0(Q)- \ell(R/Q^*) = \sum_{i=1}^{d-1} \binom{d-1}{i-1} \ell(H^i_{\frak m}(R)) + \ell(0^*_{H^d_{\frak m}(R)}),$$
where $0^*_{H^d_{\frak m}(R)}$ denotes the tight closure of the zero submodule of $H^d_{\frak m}(R)$. 
\end{remark}
The tight Hilbert function of tight Buchsbaum ring is studied by Dubey, Quy and Verma in \cite{DQV}.
\begin{theorem}
Let $(R,\fm)$ be a tight Buchsbaum local ring of dimension $d$, and $Q$ a parameter ideal. Then for all $n \ge 0$\\
$$\ell(R/(Q^{n+1})^*)=\sum_{i=0}^d(-1)^ie^*_i(Q)\binom{n+d-i}{d-i}, \text{ where }$$
{\rm (1)} $e_1^*(Q)=e_0(Q)-\ell(R/Q^*)+e_1(Q) \text{ and } e_j^*(Q)=e_j(Q)+e_{j-1}(Q) \text{ for all } 2\leq j \leq d,$\\
{\rm (2)} $e_1^*(Q)=\sum_{i=2}^{d-1} \binom{d-2}{i-2}\ell(H^i_{\frak m}(R)) + \ell(0^*_{H^d_{\frak m}(R)}),$\\
{\rm (3)} $e_i^*(Q)=(-1)^{i-1}\left[ \sum_{j=0}^{d-i}\binom{d-i-1}{j-2}\ell(H^j_{\fm}(R))+\ell(H^{d-i+1}_{\fm}(R))\right]  \text{ for } i=2, \ldots, d.$
\end{theorem}
Every tight Buchsbaum local ring is Buchsbaum. Thus the aim of this paper is to prove a generalization of the following result. 
\begin{corollary}
Let $(R, \frak m)$ be a tight Buchsbaum local ring that satisfies $(S_2)$ condition. Then $R$ is $F$-rational if and only if for some (and hence for any) parameter ideal $Q$ have $e_1^*(Q) = 0$
\end{corollary}
The distinction between Buchsbaum rings and tight Buchsbaum rings is subtle, and there exist Buchsbaum (or even Cohen-Macaulay) rings that fail to be tight Buchsbaum. For example, if $ R = \mathbb{F}_p [[X,Y, Z]]/(X^4 + Y^4 + Z^4)$ then it is a Cohen-Macaulay ring. However, $R$ is not tight Buchsbaum since, in this case, the tight closure of the zero submodule of the top local cohomology is not a $\mathbb{F}_p$-vector space.

\subsection{Limit closure of a parameter ideal} For the limit closure of parameter ideals we refer to \cite{CQ22, MQS}.
\begin{definition} Let $Q = (x_1, x_2, \ldots, x_d)$ be a parameter ideal in a Noetherian local ring $(R, \frak m)$. The {\it limit closure} of $Q$ is defined as
\[
Q^{\lim_R} = \bigcup_{n\geq 0} (x_1^{n+1}, \ldots, x_d^{n + 1}):_R (x_1\cdots x_d)^n.
\]
We will write $Q^{\lim}$ if $R$ is clear from the context.
\end{definition}
The limit closure $Q^{\lim}$ is independent of the choice of $x_1, x_2, \ldots, x_d$ since $Q^{\lim}/Q$ is the kernel of the natural map $R/Q \to H^d_{\frak m}(R)$.
\begin{remark}
\begin{enumerate}
\item (\cite[Lemma 3.3]{MQ22}) If $R$ is Buchsbaum, then for every parameter ideal $Q$ we have
$$\ell(Q^{\lim}/Q) = \sum_{i=0}^{d-1} \binom{d}{i} \ell(H^i_{\frak m}(R)).$$ 
\item Let $(R, \frak m)$ be an equidimensional local ring of characteristic $p$. We have $Q^{\lim} \subseteq Q^*$.
\end{enumerate}
\end{remark}
\section{The main results}

The following lemma follows from \cite[Theorem 8.20]{HH94} or \cite[Lemma 3.1]{DQV} that plays an important role for main result.
\begin{lemma}\label{Lem1} 
Let $(R,\fm)$ be an equidimensional excellent local ring of prime characteristic $p$ and $Q$ a parameter ideal. Then $Q^n/Q^nQ^* $ is a free $R/Q^*$-module of rank $\binom{n+d-1}{d-1}$.
\end{lemma}
\begin{lemma}\label{Lem2} 
Let $(R,\fm)$ be an equidimensional excellent local ring of prime characteristic $p$ and $Q$ a parameter ideal. Then 
$$e_1^*(Q)\geq e_0(Q)-\ell(R/Q^*)+e_1(Q).$$
\end{lemma}
\begin{proof}
For all large enough $n$ we have:
$$\ell(R/(Q^{n+1})^*)=e_0(Q)\binom{n+d}{d}-e_1^*(Q)\binom{n+d-1}{d-1}+ \text{ lower degree terms }.$$
Since $Q^nQ^*\subseteq (Q^{n+1})^*$ and by Lemma \ref{Lem1} we have 
\begin{align*}
\ell(R/(Q^{n+1})^*)&\leq \ell(R/Q^nQ^*)=\ell(R/Q^n)+ \ell(Q^n/Q^nQ^*)\\
                    &= \ell(R/Q^n)+\ell(R/Q^*)\binom{n+d-1}{d-1}.
\end{align*} 
Moreover
\begin{align*}
 &\ell(R/Q^n)\\
 &=e_0(Q)\binom{n+d-1}{d}-e_1(Q)\binom{n+d-2}{d-1}+ \cdots+ (-1)^de_d(Q)\\
 &=e_0(Q)\left(\binom{n+d}{d}-\binom{n+d-1}{d-1}\right)-e_1(Q)\left(\binom{n+d-1}{d-1}-\binom{n+d-2}{d-2}\right)+\cdots+ (-1)^de_d(Q)\\
 &= e_0(Q)\binom{n+d}{d}-(e_0(Q)+e_1(Q))\binom{n+d-1}{d-1}+ \text{ lower degree terms }. 
\end{align*}
Hence
$$\ell(R/(Q^{n+1})^*) \le e_0(Q)\binom{n+d}{d}-(e_0(Q)- \ell(R/Q^*)+e_1(Q))\binom{n+d-1}{d-1}+ \text{ lower degree terms }.$$
Therefore $e_1^*(Q)\geq e_0(Q)-\ell(R/Q^*)+e_1(Q)$.
\end{proof}
\begin{lemma}\label{Lem3} 
Let $(R,\fm)$ be a Buchsbaum local ring that satisfies $(S_2)$ condition and $Q$ a parameter ideal. Then 
$$ e_0(Q)-\ell(R/Q^\mathrm{lim})+e_1(Q)\ge 0.$$
Moreover, if the equality occurs for some $Q$ then $R$ is Cohen-Macaulay.
\end{lemma}
\begin{proof}
Since $R$ is Buchsbaum, by Theorem \ref{Hil Buchs} for all parameter ideals $Q$ we have
$$e_1(Q)=-\sum\limits_{i=1}^{d-1}\binom{d-2}{i-1}\ell(H^i_{\fm}(R)).$$
Moreover,
$$e_0(Q)-\ell(R/Q^\mathrm{lim})= \sum\limits_{i=1}^{d-1}\binom{d-1}{i-1}\ell(H^i_{\fm}(R)).$$
Therefore
$$e_0(Q)-\ell(R/Q^\mathrm{lim})+e_1(Q)=\sum\limits_{i=1}^{d-1}\binom{d-2}{i-2}\ell(H^i_{\fm}(R))\ge 0$$
for all $Q$. The equality occurs iff $H_{\fm}^i(R)=0$ for all $i=2,\ldots,d-1$. Moreover, the if ring is $(S_2)$ then $H^0_{\fm}(R)=H^1_{\fm}(R)=0$. Hence, $R$ is Cohen-Macaulay.
\end{proof}
The main result of this paper is as follows.
\begin{theorem}\label{MainTheorem} 
Let $(R,\fm)$ be an excellent reduced Buchsbaum local ring of prime characteristic $p$ that satisfies $(S_2)$ condition and $Q$ a parameter ideal. Then $R$ is $F$-rational iff $e_1^*(Q)=0$.
\end{theorem}
\begin{proof} We need only to prove that if $e_1^*(Q) = 0$ then $R$ is $F$-rational.
By Lemma \ref{Lem2}, 
$$e_1^*(Q)\ge e_0(Q)-\ell(R/Q^*)+e_1(Q).$$
Since $Q^\mathrm{lim} \subseteq Q^*$,
 $$e_0(Q)-\ell(R/Q^*)+e_1(Q) \ge e_0(Q)-\ell(R/Q^\mathrm{lim})+e_1(Q)\ge 0.$$
Thus 
$$e_0(Q)-\ell(R/Q^\mathrm{lim})+e_1(Q)=0.$$ 
By Lemma \ref{Lem3}, $R$ is Cohen-Macaulay. By Theorem \ref{CM case}, see also \cite[Corollary 4.9]{HM}, we have $R$ is $F$-rational.
\end{proof}
The main result of this paper claims that the two below questions have affirmative answers for Buchsbaum rings.
\begin{question} Let $(R, \frak m)$ be an excellent reduced and equidimensional local ring of prime characteristic $p$ that satisfies the $(S_2)$ condition. Is it true that for any parameter ideal $Q$ we have $e_0(Q)-\ell(R/Q^*)+e_1(Q) \ge 0$. Moreover, $R$ is $F$-rational iff for some (and hence for all) parameter ideal Q we have $e_0(Q)-\ell(R/Q^*)+e_1(Q) = 0$. 
\end{question}
\begin{question} Let $(R, \frak m)$ be a local ring that satisfying the $(S_2)$ condition. Is it true that $e_0(Q)-\ell(R/Q^{\lim})+e_1(Q) \ge 0$ for any parameter ideal $Q$. Moreover, $R$ is Cohen-Macaulay iff for some (and hence for all) parameter ideal $Q$ we have $e_0(Q)-\ell(R/Q^{\lim})+e_1(Q) = 0$.
\end{question}
We next show that the $(S_2)$ condition can not be removed in Theorem \ref{MainTheorem}. First, we recall some facts about $S_2$-ification of the ring by \cite{HH94}. Suppose $R$ is an unmixed local ring. We shall say that a ring $S$ is an $S_2$-ification of $R$ if it lies between $R$ and its total quotient ring, is module-finite over $R$, is
$(S_2)$ as an $R$-module, and has the property that for every element $s\in S\setminus R$, the ideal $D(s)$,
defined as $\{r \in R \mid rs \in R\}$, has height at least two. If $(R,\fm)$ is an unmixed ring then $R$ has a $S_2$-ification and this is unique (up to an isomorphism).\\
Recall that let $R, \frak m$ be a Noetherian local ring and $I$ be an ideal of $R$. An element $x \in I$ is called superficial for $I$ if there exists an integer $c$ such that for all $n \ge c$, the following equality holds:
 $$(I^{n+1}:x) \cap I^c =I^n.$$

\begin{proposition}\label{Pro1} 
Let $(R,\fm)$ be an excellent, unmixed local ring of dimension $2$, $Q=(x,y)$ a parameter ideal.  
\begin{enumerate}
\item Then
$$e_0(Q)-\ell(R/Q^\mathrm{lim})=\ell\left(\frac{H^1_{\fm}(R)}{(x,y)H^1_{\fm}(R)}\right).$$
\item If $x$ is a superficial element then we have 
$$e_0(Q)-\ell(R/Q^\mathrm{lim})+e_1(Q)=\ell\left(\frac{H^1_{\fm}(R)}{(x,y)H^1_{\fm}(R)}\right)-\ell\left(\frac{H^1_{\fm}(R)}{(x)H^1_{\fm}(R)}\right).$$
Therefore $e_0(Q)-\ell(R/Q^\mathrm{lim})+e_1(Q) \le 0$ for any parameter ideal $Q$. Moreover, if $Q H^1_{\fm}(R) = 0$ (in this case $Q$ is standard), then $e_0(Q)-\ell(R/Q^\mathrm{lim})+e_1(Q) = 0$.
\end{enumerate}
\end{proposition}
\begin{proof}
(1) The first assertion is due to \cite[Theorem 1.4]{CQ22}, however in order to be more convenient, we show again the proof of the assertion.
Let $S$ is $S_2$-ification of $R$. Set $T=S/R$, then $\mathrm{dim}(S)=\mathrm{dim}(R)=2$. Since $S$ is $(S_2)$ and $\mathrm{dim}(T)\le 0$ so $S$ is Cohen-Macaulay and $T$ is $\fm$-torsion. The short exact sequence
$$0\to R \to S \to T \to 0$$
induces the following exact sequence
$$H^0_{\fm}(S) \to H^0_{\fm}(T) \to H^1_{\fm}(R) \to H^1_{\fm}(S). $$
Hence $T=H^1_{\fm}(R)$. The short exact sequence
$$0\to R \to S \to T \to 0$$
also leads to the following short exact sequence
$$0 \to R/(QS\cap R) \to S/QS \to T/QT \to 0. $$
By \cite[ Theorem 6.2]{CQ22}, $Q^\mathrm{lim}=Q^{\mathrm{lim}_S}\cap R= QS\cap R$ since $S$ is Cohen-Macaulay.
Putting all together we have
\begin{align*}
\ell(R/Q^\mathrm{lim})&= \ell(S/QS)-\ell\left(\frac{H^1_{\fm}(R)}{QH^1_{\fm}(R)}\right)\\
&=e_0(Q,S)-\ell\left(\frac{H^1_{\fm}(R)}{QH^1_{\fm}(R)}\right)\\
&=e_0(Q,R)-\ell\left(\frac{H^1_{\fm}(R)}{QH^1_{\fm}(R)}\right).
\end{align*}

(2)
Set $R_1=R/(x)$, $\overline{R}_1=R_1/H_{\fm}^0(R_1)$ then $\mathrm{dim}(R_1)=\mathrm{dim}(\overline{R}_1)=1$, and $\overline{R}_1$ is Cohen-Macaulay. Consider the short exact sequence
$$0 \to H_{\fm}^0(R_1) \to R_1 \to \overline{R}_1 \to 0. $$
Since $H_{\fm}^0(R_1)$ is $\fm$-torsion then for large $n$
$$y^nR_1\cap H_{\fm}^0(R_1)=y^n(H_{\fm}^0(R_1):y^n)=y^nH_{\fm}^0(R_1)=0.$$
 So 
$$0 \to H_{\fm}^0(R_1) \to R_1/y^n R_1 \to \overline{R}_1/y^n\overline{R}_1 \to 0. $$
for $n\gg 0$. We have 
$$\ell(R_1/y^nR_1)-\ell(\overline{R}_1/y^n\overline{R}_1)=\ell\left( H_{\fm}^0(R_1) \right).$$
For large $n$, 
$$e_0(QR_1,R_1)\binom{n}{1}-e_1(QR_1)-e_0(Q\overline{R}_1)\binom{n}{1}=\ell(H_{\fm}^0(R_1)).$$
In other words, $e_1(QR_1, R_1)=-\ell(H_{\fm}^0(R_1))$. Because of superficial property of $x$, 
$e_1(QR_1,R_1)=e_1(Q,R)$ so $e_1(Q,R)=-\ell(H_{\fm}^0(R_1))$.
 The short exact sequence
$$0 \to R \xrightarrow{.x} R \to R_1 \to 0,$$
induces the following 
$$0 \to H_{\fm}^0(R_1) \to H^1_{\fm}(R) \xrightarrow{.x} H^1_{\fm}(R)\to \frac{H^1_{\fm}(R)}{xH^1_{\fm}(R)}\to 0. $$
Thus $$\ell\left(H^0_{\fm}(R_1)\right)=\ell\left((0:x)_{H^1_{\fm}(R)}\right)=\ell\left(\frac{H^1_{\fm}(R)}{xH^1_{\fm}(R)} \right).$$ 
Therefore, $e_1(Q,R)=-\ell\left(\frac{H^1_{\fm}(R)}{xH^1_{\fm}(R)} \right)$. Hence
$$e_0(Q)-\ell(R/Q^\mathrm{lim})+e_1(Q)=\ell\left(\frac{H^1_{\fm}(R)}{(x,y)H^1_{\fm}(R)}\right)-\ell\left(\frac{H^1_{\fm}(R)}{(x)H^1_{\fm}(R)}\right).$$
The last assertion is clear. The proof is complete.
\end{proof}
\begin{remark}
If $(R, \frak m)$ is a Buchsbaum unmixed local ring of dimension two, then $e_0(Q)-\ell(R/Q^\mathrm{lim})+e_1(Q) = 0$ for any parameter ideal $Q$.
\end{remark}
We provide some concrete examples. The first example is closely related to Proposition 5.3 of \cite{DQV}. 
\begin{example}
Let $S=k[[X,Y,Z,T]]$ be the formal power series of four variables. Let $I=(X,Y)^2$, $J=(Z,T)^2$ and $R=S/I\cap J$. Let $x, y, z, w$ be the image of $X, Y, Z, W$ in $R$. Notice that $S/I$ and $S/J$ are Cohen-Macaulay of dimension $\mathrm{dim}(S/I)=\mathrm{dim}(S/J)=2$. Since $R$ is unmixed of dimension $\mathrm{dim}(R)=2$, so $R$ is generalized Cohen-Macaulay. Let $\fm=(x,y,z,t)R$ be the maximal ideal of $R$.
The short exact sequence $$0 \to S/(I\cap J) \to S/I\oplus S/J \to S/(I+J) \to 0$$ induces the following exact sequence
$$H^0_{\fm}(S/I)\oplus H^0_{\fm}(S/J)\to H^0_{\fm}(S/(I+J)) \to H^1_{\fm}(S/I\cap J)\to H^1_{\fm}(S/I)\oplus H^1_{\fm}(S/J).$$
Since $S/I$ and $S/J$ are Cohen-Macaulay of dimension $2$ we have
$$H^1_{\fm}(R)\cong H^0_{\fm}(S/(I+J)) \cong S/(I+J)=k[[X,Y,Z,T]]/((X,Y)^2+(Z,T)^2).$$
Hence $\fm H^1_{\fm}(R)\neq 0$, the ring is not Buchsbaum.
It is not difficult to check that $Q=(x+z,y+t)R$ is a parameter ideal of $R$. Let $a$ be a superficial element of $R$ with respect to $Q$. Notice that $\ell(S/(I + J + (a))) \ge 4$. Therefore, by Proposition \ref{Pro1} we have 
$$e_0(Q)-\ell(R/Q^\mathrm{lim})+e_1(Q)=\ell\left(\frac{H^1_{\fm}(R)}{QH^1_{\fm}(R)}\right)-\ell\left(\frac{H^1_{\fm}(R)}{aH^1_{\fm}(R)}\right)\le 3-4< 0.$$
Hence, the $(S_2)$ condition is important for Lemma \ref{Lem3}, and hence for Theorem \ref{MainTheorem}.
\end{example}
\begin{example}
Let $R=\mathbb{F}_p[[X^5,X^4Y,XY^4,Y^5]]$ with the maximal ideal $\fm=(X^5,X^4Y,XY^4,Y^5)$ whose $S_2$-ification is $S=\mathbb{F}_p[[X^5,X^4Y,X^3Y^2,X^2Y^3,XY^4,Y^5]]$, $\mathrm{dim}(R)=\mathrm{dim}(S)=2$. It should be note that $S$ is a Veronese subring of the local regular ring $\mathbb{F}_p[[X,Y]]$,  $S$ is $F$-rational so that $H^0_{\fm}(S)=H^1_{\fm}(S)=0$ and $0^*_{H^2_{\fm}(S)}=0$. Set $N=S/R$ then 
$$N\cong \mathbb{F}_pX^2Y^3 \oplus \mathbb{F}_pX^3Y^2 \oplus \mathbb{F}_pX^3Y^7 \oplus \mathbb{F}_pX^7Y^3.$$
The short exact sequence $0\to R \to S \to N \to 0$ induces the following exact sequence
$$0\to H^0_{\fm}(R) \to H^0_{\fm}(S) \to H^0_{\fm}(N) \to H^1_{\fm}(R) \to H^1_{\fm}(S) \to $$
$$\to H^1_{\fm}(N)\to H^2_{\fm}(R)\to H^2_{\fm}(S)\to H^2_{\fm}(N).$$
Since $N$ has finite length so that it is $\fm$-tosion and $S$ is $F$-rational then 
$$H^0_{\fm}(R)=0,\quad H^1_{\fm}(R)\cong H^0_{\fm}(N)\cong N,\quad  H^2_{\fm}(R)\cong H^2_{\fm}(S).$$
Moreover, $\fm H^1_{\fm}(R)\neq 0$ since $X^5.X^2Y^3=X^7Y^3\neq 0$ in $H^1_{\fm}(R)$. Hence $R$ is not Buchsbaum but it is generalized Cohen-Macaulay. We have $f=aX^5+bY^5$ with $a,b\in \mathbb{F}_p$ and $a,b\neq 0$ is a superficial element of $R$ and $Q=(X^5,Y^5)=(f, Y^5)$ is a parameter ideal of $R$. It is obvious that 
$$(f)H^1_{\fm}(R) \cong \mathbb{F}_pX^7Y^3 \oplus \mathbb{F}_pX^3Y^7 \cong QH^1_{\fm}(R).$$
Since $Q^*/Q^{\mathrm{lim}}\subseteq 0^*_{H^2_{\fm}(R)} \cong 0^*_{H^2_{\fm}(S)} \cong 0$ and by Proposition \ref{Pro1},
\begin{align*}
e_0(Q)-\ell(R/Q^*)+e_1(Q)&=e_0(Q)-\ell(R/Q^\mathrm{lim})+e_1(Q)\\
&=\ell\left(\frac{H^1_{\fm}(R)}{QH^1_{\fm}(R)}\right)-\ell\left(\frac{H^1_{\fm}(R)}{(f)H^1_{\fm}(R)}\right)\\
&=2-2=0.
\end{align*}
Moreover $S$ is a module finite over $R$, we have $(Q^n)^* = (Q^nS)^* \cap R$. On the other hand, $S$ is $F$-rational (in fact, $F$-regular), we have 
$(Q^n)^* = Q^nS \cap R$ for all $n$. Thus we get the following short exact sequence 
$$ 0 \to R/(Q^{n+1})^* \to S/Q^{n+1}S \to N \to 0$$
for all $n \gg 0$. Hence $\ell(R/(Q^{n+1})^*) = e_0(Q) \binom{n+2}{2} - \ell(N) = e_0(Q) \binom{n+2}{2} - 4$. Hence $e_1^*(Q) = 0$ but $R$ is not $F$-rational.
\end{example}

\begin{remark} The main results of this paper may hold in mixed characteristic when $e_1^*(Q)$ is replaced with $e_1^B(Q)$ for any big Cohen-Macaulay ring $B$ of $R$. Additionally, it would be interesting to study the analogue of tight Buchsbaum rings in mixed characteristic which we refer to as {\it BCM-Buchsbaum rings}.
\end{remark}

\end{document}